\documentclass[12pt,reqno]{amsart}
\usepackage{latexsym,amssymb,amsmath,multicol}
 2

\newtheorem{thm}{Theorem}[section]

\newtheorem{cor}[thm]{Corollary}
\newcommand{\thmref}[1]{Theorem~\ref{#1}}

\newcommand{\corref}[1]{Corollary~\ref{#1}}

\theoremstyle{remark}
\newtheorem{rmk}{Remark}[section]

\begin{document}

\title[Representation of certain quadratic forms]
{On the number of representations of certain quadratic forms and a formula for the Ramanujan Tau function}
 
\author{B. Ramakrishnan, Brundaban Sahu and Anup Kumar Singh}
\address[B. Ramakrishnan and Anup Kumar Singh]{Harish-Chandra Research Institute, HBNI,  
       Chhatnag Road, Jhunsi,
     Allahabad -     211 019,
   India.}
\address[Brundaban Sahu]
{School of Mathematical Sciences, National Institute of Science 
Education and Research, Bhubaneswar, HBNI, 
Via- Jatni, Khurda, Odisha 752 050  
India.}

\email[B. Ramakrishnan]{ramki@hri.res.in}
\email[Brundaban Sahu]{brundaban.sahu@niser.ac.in}
\email[Anup Kumar Singh]{anupsingh@hri.res.in}

\dedicatory{Dedicated to Srinivasa Ramanujan on the occasion of his 129th 
birth anniversary}

\date{\today}

\begin{abstract}

In this paper, we find the number of representations of the quadratic form $x_1^2+ x_1x_2 + x_2^2 + x_3^2+ x_3x_4 + x_4^2 +
\ldots + x_{2k-1}^2 + x_{2k-1}x_{2k} + x_{2k}^2,$  for  $k=7,9,11,12,14$ 
using the theory of modular forms.  By comparing our formulas with the formulas obtained by G. A. Lomadze, we obtain the Fourier coefficients of certain newforms of level $3$ 
and weights $7,9,11$ in terms of certain finite sums involving the solutions of similar quadratic forms of lower variables. In the case of $24$ variables, comparison of these 
formulas gives rise to a new formula for the Ramanujan tau function. 
\end{abstract}


\maketitle

\section{Introduction}

For a positive integer $k$, let $F_k$ denote the quadratic form in $2k$ variables defined by 
\begin{equation}\label{quad-form}
F_k(x_1, x_2, \ldots,x_{2k}) =  \sum_{j=1}^k x_{2j-1}^2 + x_{2j-1} x_{2j} + x_{2j}^2,
\end{equation}
and we denote by 
\begin{equation}\label{s2k}
s_{2k}(n) = {\rm card}\left\{(x_1, x_2, \cdots, x_{2k}) \in {\mathbb Z}^{2k} : F_k(x_1, x_2, \cdots, x_{2k}) = n \right\},
\end{equation}
the number of representations of a positive integer $n$ by the quadratic form $F_k$. Finding explicit formulas for $s_{2k}(n)$ is a classical 
problem. For $k=2,4,6,8, 10, 12$ formulas for $s_{2k}$ are known due to the works of J. Liouville \cite{liouville},  J. G. Huard et. al. \cite{huard},  
O. X. M. Yao and E. X. W. Xia \cite{xia} and the first two authors \cite{{ramki1}, {ramki2}}. Most of these works make use of the convolution sums 
of the divisor functions to evaluate the formulas. In \cite{lomadze}, G. A.  Lomadze gave formulas for $s_{2k}(n)$ for 
$2\le k\le 17$. These formulas were given in terms of divisor functions and certain finite sums involving polynomials in $x_1$ ($F_l(x_1, x_2, \ldots, x_{2l}) =n$) with integer coefficients (including $n$) for certain values of $l<k$. 
However, the other formulas mentioned above are in terms of the divisor functions and Fourier coefficients of certain cusp forms.  So, 
it is natural to compare these formulas with the formulas given by Lomadze and this gives rise to certain interesting formulas for the Fourier coefficients of 
certain cuspidal newforms of integral weight. See for example \cite[Theorem 1.3]{xia}, \cite[Theorem 2.3]{ramki1}, \cite[Corollary 2.5]{ramki2}. 

The purpose of this paper is to consider the cases $k= 7,9,11,12,14$ and find the formulas for $s_{2k}(n)$ in these cases (\thmref{odd}, \thmref{s28}). We make use of the 
theory of modular forms directly instead of using the convolution sums method. When comparing our formulas with the formulas given by Lomadze, we 
conclude that the finite sums appearing in the formulas (VI), (VIII) and (X) of Lomadze \cite[p.19]{lomadze} correspond to the $n$-th Fourier coefficients 
(up to some constant) of the newforms of weights $7,9,11$ with level $3$ and character $\chi_{-3} = \left(\frac{-3}{\cdot}\right)$. 

In an earlier work of the first two authors \cite{ramki2}, the case of $24$ variables was treated and as a consequence they obtained a formula involving the Ramanujan 
function $\tau(n)$ and the Fourier coefficients of the newform of weight $12$ on $\Gamma_0(3)$.  In the present work, we make use of a different basis for the modular forms 
space of weight $12$ on $\Gamma_0(3)$ which results in a different formula for $s_{24}(n)$ (see \eqref{24}). The advantage of considering this basis is that while comparing the new 
formula of $s_{24}(n)$ with Lomadze's formula, we get a new identity for $\tau(n)$ in terms of finite sums that appear in Lomadze's formulas (\corref{tau}). In the case of 
$28$ variables, comparison of the formula for $s_{28}(n)$ obtained in \eqref{28} with Lomadze's formula leads to a relation between these finite sums (coming from Lomadze's formulas) and certain convolution sums of the divisor functions (\corref{conv}).  
In principle, one can adopt our method to get formulas for $s_{2k}(n)$ for other higher values of $k$ and get similar identities between the Fourier 
coefficients of newforms of the corresponding integral weight and the sums appearing in Lomadze's formulas. Our aim in this paper is to fill up the 
incomplete cases (corresponding to the odd weight) and also to get a new expression for the Ramanujan tau function.



\section{Preliminaries and Statement of the Results}

As mentioned in the introduction, we shall be using the theory of modular forms to prove our results and so we first fix our notations and present 
some of the basic facts on modular forms. 
For positive  integers $k, N \ge 1$ and a Dirichlet character $\chi$ modulo $N$ with $\chi(-1) = (-1)^k$, let  
$M_k(N, \chi)$ denote the ${\mathbb C}$- vector space of holomorphic modular forms of weight $k$ for the congruence subgroup $\Gamma_0(N)$, with character $\chi$. Let us denote by $S_k(N, \chi)$,  the subspace of cusp forms in $M_k(N, \chi)$. When $\chi$ is the principal character modulo $N$, then we 
drop the symbol $\chi$ in the notation and write only $M_k(N)$ or $S_k(N)$.  The modular forms space is decomposed into the space of 
Eisenstein series (denoted by ${\mathcal E}_k(N, \chi)$) and the space of cusp forms $S_k(N, \chi)$ and one has 
$$
M_k(N, \chi) =  {\mathcal E}_k(N, \chi) \oplus S_k(N, \chi).
$$
One can get an explicit basis of the space ${\mathcal E}_k(N, \chi)$ using the following construction. For details we refer to \cite{{miyake}, {stein}}. 
Suppose that $\chi$ and $\psi$ are primitive Dirichlet characters with conductors $N$ and $M$, respectively. For a positive integer $k>2$, let 
\begin{equation}\label{eisenstein}
E_{k,\chi,\psi}(z) :=  c_0 + \sum_{n\ge 1}\left(\sum_{d\vert n} \psi(d) \cdot \chi(n/d) d^{k-1}\right) q^n,
\end{equation}
where $q=e^{2 i\pi z} ~(z\in {\mathcal H})$,  
$$
c_0 = \begin{cases}
0 &{\rm ~if~} N>1,\\
- \frac{B_{k,\psi}}{2k} & {\rm ~if~} N=1,
\end{cases}
$$
and $B_{k,\psi}$ denotes generalized Bernoulli number with respect to the character $\psi$. 
Then, the Eisenstein series $E_{k,\chi,\psi}(z)$ belongs to the space $M_k(NM, \chi/\psi)$, provided $\chi(-1)\psi(-1) = (-1)^k$ 
and $NM\not=1$. We give a notation to the inner sum in \eqref{eisenstein}:
\begin{equation}\label{divisor}
\sigma_{k-1;\chi,\psi}(n) := \sum_{d\vert n} \psi(d) \cdot \chi(n/d) d^{k-1}.
\end{equation}
When $\chi=\psi =1$ (i.e., when $N=M=1$) and $k\ge 4$, we have $E_{k,\chi,\psi}(z) = - \frac{B_k}{2k} E_k(z)$, where $E_k$ is the normalized Eisenstein series of weight $k$ in the space $M_k(1)$, defined  by 
$$
E_k(z) = 1 - \frac{2k}{B_k}\sum_{n\ge 1} \sigma_{k-1}(n) q^n. 
$$
In the above $\sigma_r(n)$ is the sum of the $r$th powers of the positive divisors of $n$ and $B_k$ is the $k$-th Bernoulli number defined by $\displaystyle{\frac{x}{e^x-1} = \sum_{m=0}^\infty \frac{B_m}{m!} x^m}$.
We also need the Eisenstein series of weight $2$, which is a  quasimodular form of weight $2$, depth $1$ on $SL_2({\mathbb Z})$ and is given by 
$$
E_2(z) = 1 -24\sum_{n\ge 1} \sigma(n) q^n. 
$$
(Here $\sigma(n) = \sigma_1(n)$.) Let $\Delta(z) = \sum_{n\ge 1} \tau(n) q^n $ be the well-known unique normalized cusp form of weight $12$, level $1$, studied by Ramanujan \cite{ramanujan}. It is known that $\Delta(z) = \eta^{24}(z)$, where $\eta(z)$ is the Dedekind eta function given by 
$$
\eta(z)=q^{1/24} \prod_{n\ge1}(1-q^n).
$$
In the case of the space of cusp forms $S_k(N,\chi)$ we use a basis consisting of newforms of level $N$ and oldforms generated by the newforms of lower level $d$, $d\vert N$, $\chi$ modulo $d$, $d\not =N$. However, in the case of $k=12$, we use a different basis for $S_{12}(3)$ (see \S 3.2). 
For basic theory of newforms we refer to \cite{{a-l},{li}} and for details on modular forms and quasimodular forms, we refer to \cite{{1-2-3}, {koblitz},  {miyake}, {stein}}. 

\smallskip

We now state the main results of this paper. 

Let $s_{2k}(n)$ be as defined in \eqref{s2k}. Then we have the following formulas for $s_{2k}(n)$ for $k=7,9,11$.

\begin{thm}\label{odd}
For a positive integer $n$, we have 
\begin{eqnarray}
s_{14}(n) & = & \frac{3}{7} \rho_6^*(n) + \frac{216}{7} \tau_{7,3,\chi_{-3}}(n), \label{s14}\\
s_{18}(n) & = & \frac{27}{809} \rho_8^*(n) + \frac{24 \times 1728}{809}\left( 27 
\tau_{9,3,\chi_{-3};1}(n) + \tau_{9,3,\chi_{-3};2}(n)\right), \label{s18} \\
s_{22}(n) & = & \frac{3}{1847} \rho_{10}^*(n) + \frac{81\times 748}{9235} \left( 
\tau_{11,3,\chi_{-3};1}(n) + 9 \tau_{11,3,\chi_{-3};2}(n) \right), \label{s22}
\end{eqnarray}
where $\rho_\ell^*(n)$ (for positive even integers $\ell$) is defined by 
\begin{equation*}
\rho_\ell^*(n) = 3^{\ell/2} \sum_{d\vert n} \left(\left(\frac{n/d}{3}\right) + (-1)^{\ell/2} \left(\frac{d}{3}\right)\right) d^\ell,
\end{equation*}
$\tau_{7,3,\chi_{-3}}(n)$ is the $n$th Fourier coefficient of the normalized newform of weight $7$, level $3$ with character $\chi_{-3}$ and for $k=9,11$, 
$\tau_{k,3,\chi_{-3};j}(n)$ ($j=1,2$) are the $n$-th Fourier coefficients of basis elements of the space $S_k(3,\chi_{-3})$. (These are defined explicitly in \S 3.1.)
\end{thm}

As mentioned in the introduction, by comparing the corresponding formulas obtained by 
Lomadze \cite{lomadze} with our formulas of the above theorem, as a direct consequence, we get the following corollary. 

\begin{cor}
The Fourier coefficients of the newforms of weights $7,9,11$ (of level $3$ with character $\chi_{-3}$) are given by the following sums, 
which involve the first coordinate ($x_1$) of the solutions to the quadratic forms $F_j(x_1,x_2,\cdots, x_{2j}) =n$, $j=3,5,7$ respectively. More precisely, we have 
\begin{equation}
\begin{split}
\tau_{7,3,\chi_{-3}}(n) & = \frac{1}{30} \sum_{F_3(x_1,\cdots,x_6)=n} (15 x_1^4 -12n x_1^2 + n^2),\\
(27 \tau_{9,3,\chi_{-3};1}(n) + \tau_{9,3,\chi_{-3};2}(n)) & =\frac{1}{168} \sum_{F_5(x_1,\cdots,x_{10})=n} (63 x_1^4 -36n x_1^2 + 2n^2),\\
(\tau_{11,3,\chi_{-3};1}(n) + 9 \tau_{11,3,\chi_{-3};2}(n)) & =\frac{5}{81} \sum_{F_7(x_1,\cdots,x_{14})=n} (54 x_1^4 -24 n x_1^2 + n^2).\\
\end{split}
\end{equation}
\end{cor}

For the cases $k=12, 14$, we have the following theorem giving formulas for $s_{24}(n)$ and $s_{28}(n)$. As mentioned earlier, a different formula for $s_{24}(n)$ was 
given in \cite[Theorem 2.4]{ramki2}. 

\begin{thm}\label{s28}
For a positive integer $n$, we have 
\begin{equation}\label{24}
\begin{split}
s_{24}(n)&= \frac{6552}{73\times 691}\sigma_{11}^*(n) + \frac{29824}{691} \tau(n) +  \frac{240\times 1186848}{50443} \sum_{a,b\in 
{\mathbb N}_0\atop{a+b=n}} \sigma_3(a) \tau_{8,3}(b)\\
&\quad - \frac{504\times 261344}{50443} \sum_{a,b\in {\mathbb N}_0\atop{a+b=n}} \sigma_5(a) \tau_{6,3}(b), 
\end{split}
\end{equation}

\begin{equation}\label{28}
\begin{split}
s_{28}(n) & = \frac{12}{1093} \sigma_{13}^*(n) + \frac{107264}{1093} \tau(n) + \frac{107264 \times 12}{1093} \left(\sum_{a,b\in {\mathbb N}\atop{a+b=n}}\!\!\!\sigma(a) \tau(b) - 3 \sum_{a,b\in {\mathbb N}\atop {3a+b=n}}\!\!\!\sigma(a)\tau(b)\right)\\
& \quad + \frac{12448\times 504}{1093}\sum_{a,b\in {\mathbb N}_0 
\atop{a+b=n}} \!\!\!\sigma_5(a) \tau_{8,3}(b)
- \frac{3016\times 480}{1093} \sum_{a,b \in {\mathbb N}_0\atop{a+b=n}} \!\!\!\sigma_7(a) \tau_{6,3}(b),
\end{split}
\end{equation}
where $\sigma_{\ell}^*(n) = \sigma_{\ell}(n) + (-3)^{(\ell+1)/2} \sigma_{\ell}(n/3)$, ${\mathbb N}_0 = {\mathbb N}\cup \{0\}$, $\sigma_3(0) = 1/240$, 
$\sigma_5(0) = -1/504$, $\sigma_7(0) = 1/480$ and $\tau_{6,3}(0) = 0 = \tau_{8,3}(0)$. Also,  $\tau(n)$ is the Ramanujan tau function as defined before, and for $k=6,8$, $\tau_{k,3}(n)$ is the $n$th Fourier coefficient of $\Delta_{k,3}(z)$, the normalized newform of weight $k$, level $3$ with trivial character. 
\end{thm}

\noindent {\bf Note:}  All the cusp forms mentioned in \thmref{odd} and \thmref{s28} are defined in sections 3.1 and 3.2 respectively.  

As a consequence to our formula for $s_{24}(n)$ given by \eqref{24}, we have the  following formula for the Ramanujan tau function: 
\begin{cor}\label{tau}
For an integer $n\ge 1$, we have 
\begin{equation}\label{tau-eq}
\begin{split}
\tau(n) & = \frac{1}{73\times 3728}\Big[\frac{36387}{35} L_{12;8}(n) + 108 L_{12;6}(n) + \frac{1}{3} {\mathcal L}_4(n)  - \frac{32668}{12} L_{6;2}(n)\\
& \quad  - 329680 \sum_{a,b \in {\mathbb N}\atop{ a+b=n}} \sigma_3(a) L_{8;4}(b) + 1372056 \sum_{a,b \in {\mathbb N}\atop{ a+b=n}} \sigma_5(a) L_{6;2}(b)\Big], \\
\end{split}
\end{equation}
where 
\begin{eqnarray}
L_{12;8}(n) &= & \sum_{x_1\in {\mathbb Z}\atop{F_{8}(x_1,\ldots,x_{16})=n}} 135 x_1^4 - 54 n x_1^2 + 2 n^2, \label{12-8}\\
L_{12;6}(n) & = & \sum_{x_1\in {\mathbb Z}\atop{F_{6}(x_1,\ldots,x_{12})=n}} 162 x_1^6 - 162 n x_1^4 + 36 n^2 x_1^2- n^3,\label{12-6}\\
L_{8;4}(n) &= & \sum_{x_1\in {\mathbb Z}\atop{F_{4}(x_1,\ldots,x_{8})=n}} 45 x_1^4 - 30 n x_1^2 + 2 n^2,  \\
L_{6;2}(n) & = & \sum_{x_1\in {\mathbb Z}\atop{F_{2}(x_1,\ldots,x_{4})=n}} 9 x_1^4 - 9 n x_1^2 + n^2,\\
{\mathcal L}_4(n) & = &  \sum_{x_1\in {\mathbb Z}\atop{F_{4}(x_1,\ldots,x_{8})=n}} 
\!\!\!\!\! \big(164025  x_1^8 - 306180 n x_1^6 + 45(3780 n^2 -4121) x_1^4 \nonumber \\
& &\hskip 3cm  -  30(945 n^2 - 4121) n x_1^2 + 675 n^4 - 8242 n^2\big).
\end{eqnarray}
\end{cor}

\begin{proof}
In \cite[formula (XI), p.12]{lomadze}, Lomadze gave the following formula for $s_{24}(n)$:
\begin{equation}\label{L-24}
s_{24}(n) = \frac{1}{73\times 691}\Big({6552} \sigma_{11}^*(n) + \frac{291096}{35} L_{12;8}(n) + {864} L_{12;6}(n) + {360} L_{12;4}(n)\Big),
\end{equation}
where $L_{12;8}(n)$ and $L_{12;6}(n)$ are given by \eqref{12-8} and \eqref{12-6} respectively, and $L_{12;4}(n)$ is given by 
\begin{equation}
L_{12;4}(n) = \sum_{x_1\in {\mathbb Z}\atop{F_{4}(x_1,\ldots,x_{8})=n}} 1215 x_1^8 - 2268 n x_1^6 + 1260 n^2 x_1^4 - 210 n^3 x_1^2 + 5 n^4.
\end{equation}
Comparing this with \eqref{24}, we get the following relation (after cancelling the factor $50443$ in the denominators).
\begin{equation}\label{tau-1}
\begin{split}
&29824 \times 73 \tau(n) + 1186848  \tau_{8,3}(n)  + 261344 \tau_{6,3}(n) \hskip 2cm \\
&\quad + 1186848 \times 240 \sum_{a,b \in {\mathbb N}\atop{ a+b=n}} \sigma_3(a) \tau_{8,3}(b) 
- 261344 \times 504 \sum_{a,b \in {\mathbb N}\atop{ a+b=n}} \sigma_5(a) \tau_{6,3}(b)   \hskip 2cm \\
& \qquad = \frac{291096}{35} L_{12;8}(n) + {864} L_{12;6}(n) + {360} L_{12;4}(n).
\end{split}
\end{equation}
Now, in \cite[Theorem 1.3]{xia}, Yao and Xia showed that 
\begin{equation}\label{tau-2}
\tau_{6,3}(n) =  \frac{1}{12} L_{6;2}(n)
\end{equation}
and in \cite[Theorem 2.3]{ramki1}  the first two authors obtained the following expression for the newform Fourier coefficients $\tau_{8,3}(n)$:
\begin{equation}\label{tau-3}
\tau_{8,3}(n)  =   \frac{1}{108} L_{8;4}(n). 
\end{equation}
Substituting \eqref{tau-2} and \eqref{tau-3} in \eqref{tau-1}, we get the required formula. Note that both the sums $L_{8;4}(n)$ and $L_{12;4}(n)$ involve the solutions of 
the quadratic form $F_4$, and so these are combined to get the expression ${\mathcal L}_4(n)$. 
\end{proof}

In the following corollary, we shall obtain a relation between certain convolution sums in terms of some of the finite sums appearing in Lomadze's formulas as a consequence 
to our formula for $s_{28}(n)$. 

\begin{cor}\label{conv}
For an integer $n\ge 1$, we have the following identity.
\begin{equation}\label{s28-cor}
\begin{split}
73760 \sum_{a,b \in {\mathbb N}\atop{ a+b=n}} \sigma_7(a) L_{6;2}(b) - \frac{194432}{3} \sum_{a,b \in {\mathbb N}\atop{ a+b=n}} \sigma_5(a) L_{8;4}(b) + 60336 \sum_{a,b \in {\mathbb N}\atop{ a+b=n}} \sigma_3(a) L_{10;6}(b)  \hskip 2cm & \\
 \qquad = \frac{461}{3} L_{6;2}(n) - \frac{3472}{27} L_{8;4}(n) - \frac{1257}{5} L_{10;6}(n) + \frac{94477}{735} L_{14;10}(n) \hskip 2cm &\\
 \qquad \qquad + \frac{864}{245} L_{14;8}(n) + \frac{144}{175} L_{14;6}(n). \hskip 6cm &\\
\end{split}
\end{equation}
\end{cor}

\begin{proof}
The following is the formula for $s_{28}(n)$ obtained by Lomadze \cite[formula (XIII), p13]{lomadze}: 
\begin{equation}\label{L-28}
s_{28}(n) = \frac{12}{1093} \sigma_{13}^*(n) + \frac{188954}{803355} L_{14;10}(n) + \frac{1728}{267785} L_{14;8}(n) + \frac{288}{191275} L_{14;6}(n),
\end{equation}
where 
\begin{eqnarray}
L_{14;10}(n) & = & \!\!\!\sum_{x_1\in {\mathbb Z}\atop{F_{10}(x_1,\ldots,x_{20})=n}} \!\!\!\!\!
(99x_1^4 - 33 n x_1^2 + n^2), \label{14-10}\\
L_{14;8}(n) & = & \!\!\!\sum_{x_1\in {\mathbb Z}\atop{F_{8}(x_1,\ldots,x_{16})=n}} \!\!\!\!\! (594 x_1^6 - 495 n x_1^4 + 90n^2x_1^2 - 2n^3),\label{14-8} \\
L_{14;6}(n) & = & \!\!\!\sum_{x_1\in {\mathbb Z}\atop{F_{6}(x_1,\ldots,x_{12})=n}}\!\! \!\!\!(8019 x_1^8 - 12474 n x_1^6 + 5670n^2x_1^4 - 756 n^3 x_1^2+ 14n^4). \label{14-6} 
\end{eqnarray}
Comparing this with \eqref{28}, we get the following.
\begin{equation}\label{new-1}
\begin{split}
107264 \big[\tau(n)  + 12 \sum_{a,b \in {\mathbb N}, a+b=n}\sigma(a) \tau(b) - 36 \sum_{a,b \in {\mathbb N}, 3a+b=n}\sigma(a) \tau(b)\big] - 12448 \tau_{8,3}(n) \qquad \qquad &\\
- 3016 \tau_{6,3}(n) + 6273792 \sum_{a,b \in {\mathbb N}, a+b=n}\sigma_5(a) \tau_{8,3}(b) - 1447680 
\sum_{a,b \in {\mathbb N}, a+b=n}\sigma_7(a) \tau_{6,3}(b)  \qquad &\\
~=~\quad \frac{188954}{735} L_{14;10}(n) + \frac{1728}{245} L_{14;8}(n) + \frac{288}{175} L_{14;6}(n).\hskip 3cm  &
\end{split}
\end{equation}
The following is a well-known convolution sum of $\sigma(n)$ with $\tau(n)$ due to Ramanujan \cite[Eq.(99)]{ramanujan}: 
\begin{equation}\label{tau1}
\sum_{a,b \in {\mathbb N}, a+b=n} \sigma(a) \tau(b) = \frac{1}{24}(1-n) \tau(n).
\end{equation}
Considering the quasimodular form $E_2(3z) \Delta(z)$, we obtain the following convolution sum: 
\begin{equation}\label{tau2}
\begin{split}
\sum_{a,b \in {\mathbb N}, 3a+b=n} \sigma(a) \tau(b) & = \frac{(3 - n)}{72} \tau(n) -\frac{1}{576} \tau_{6,3}(n) -\frac{1}{96}\tau_{8,3}(n) - \frac{1}{64} \tau_{10,3;2}(n) \\
& ~~- \frac{5}{6} \sum_{a,b \in {\mathbb N}\atop{ a+b=n}} \!\!\!\!\sigma_7(a) \tau_{6,3}(b) + \frac{21}{4} \sum_{a,b \in {\mathbb N}\atop {a+b=n}} \!\!\!\!\sigma_5(a) \tau_{8,3}(b) - \frac{15}{4} \sum_{a,b \in {\mathbb N} \atop{a+b=n}} \!\!\!\!\sigma_3(a) \tau_{10,3;2}(b).  
\end{split}
\end{equation}
In \cite[Corollary 2.5]{ramki2}, the first two authors obtained an expression for the Fourier coefficients of the newform $\Delta_{10,3;2}(z)$, which is given below. 
\begin{eqnarray}
\tau_{10,3;2}(n) &=& \frac{1}{120} \sum_{x_1\in {\mathbb Z}\atop{F_6(x_1, \ldots, x_{12})=n}} 
(42 x_1^4 - 21n x_1^2 + n^2) =: \frac{1}{120} L_{10;6}(n).\label{10-3}
\end{eqnarray}
Substituting the expressions \eqref{tau1}, \eqref{tau2} in \eqref{new-1}, we get 
\begin{equation}
\begin{split}
1844 \tau_{6,3}(n) + 13888 \tau_{8,3}(n) + 30168 \tau_{10,3;2}(n) + 885120 \sum_{a,b \in {\mathbb N}, a+b=n}\sigma_7(a) \tau_{6,3}(b) &\\
\qquad - 1999552 \sum_{a,b \in {\mathbb N}, a+b=n}\sigma_5(a) \tau_{8,3}(b) + 7240320 \sum_{a,b \in {\mathbb N}, a+b=n}\sigma_3(a) \tau_{10,3;2}(b) \hskip 1cm &\\
\qquad \quad =  \frac{94477}{735} L_{14;10}(n) + \frac{864}{245} L_{14;8}(n) + \frac{144}{175} L_{14;6}(n).\hskip 1cm  &\\
\end{split}
\end{equation}
Now using the expressions \eqref{tau-2}, \eqref{tau-3} and \eqref{10-3} in the above and simplifying, we get the required formula.
\end{proof}

\begin{rmk}
{\rm We note that as done in \corref{tau}, one can combine the sums $L_{10;6}(n)$ and $L_{14;6}(n)$  in the identity \eqref{s28-cor}.
We also note that the sum that appears in \eqref{10-3} is the correct one. In Lomadze's formula \cite[formula (IX), p. 12]{lomadze}, the coefficient of $n x_1^2$ is 
wrongly mentioned as $27$ instead of $21$. The same mistake also appears in \cite[Corollary 2.5]{ramki2}.
}
\end{rmk}

\section{Proofs of theorems}

\subsection{Proof of \thmref{odd}}

We denote the theta series associated to the quadratic form $F_1: x_1^2+x_1x_2+x_2^2$ by 
\begin{equation}\label{F1}
{\mathcal F_1}(z) = \sum_{x_1,x_2\in {\mathbb Z}} q^{x_1^2+x_1x_2+x_2^2}.
\end{equation}
By \cite[Theorem 4]{schoeneberg}, it follows that ${\mathcal F_1}(z)$ is a modular form in 
$M_1(\Gamma_0(3),\chi_{-3})$. 
Note that the modular form ${\mathcal F}_k$ associated to the quadratic form $F_k$ is nothing but ${\mathcal F}_1^k$. Therefore, by the definition of $s_{2k}(n)$ (Eq.\eqref{s2k}), we have  
\begin{equation*}
{\mathcal F}_k(z) = {\mathcal F}_1^k(z) = \sum_{n\ge 0} s_{2k}(n) q^n.
\end{equation*}
So, in order to get the required formulas $s_{2k}(n)$, $k=7,9,11$, we need to find explicit bases for the vector spaces $M_k(\Gamma_0(3),\chi_{-3})$ for 
$k=7,9,11$. \\
\noindent Case (i): $k=7$.  
A basis for the $3$-dimensional vector space $M_7(\Gamma_0(3), \chi_{-3})$ is given by $E_{7, {\bf 1},\chi_{-3}}(z)$, 
$E_{7, \chi_{-3}, {\bf 1}}(z)$ and the unique normalized newform $\Delta_{7,3,\chi_{-3}}(z)$ is given by 
$$
\Delta_{7,3,\chi_{-3}}(z) = \sum_{n\ge 1} \tau_{7,3,\chi_{-3}}(n) q^n ~=~ \big(E_4(z) - E_4(3z)\big) \eta^9(z) \eta^{-3}(3z).  
$$
By comparing the first few Fourier coefficients (using the Sturm bound), it follows that 
$$
{\mathcal F}_7(z) = \frac{81}{7} E_{7,{\bf 1},\chi_{-3}}(z) - \frac{3}{7} E_{7, \chi_{-3}, {\bf 1}}(z) + \frac{216}{7} \Delta_{7,3,\chi_{-3}}(z),
$$
from which we get the formula for $s_{14}(n)$ given below. 
$$
s_{14}(n) = \frac{3}{7} \rho_6^*(n) + \frac{216}{7} \tau_{7,3,\chi_{-3}}(n).
$$

\noindent Case (ii): $k=9$.  
In this case,  the dimension of the corresponding modular forms space is $4$. A basis for the $4$-dimensional vector space $M_9(\Gamma_0(3), \chi_{-3})$ is given by $E_{9, {\bf 1},\chi_{-3}}(z)$, $E_{9, \chi_{-3}, {\bf 1}}(z)$, $\Delta_{9,3,\chi_{-3};j}(z)$, $j=1,2$, where 
\begin{equation*}
\begin{split}
\Delta_{9,3,\chi_{-3};1}(z) := \sum_{n\ge 1} \tau_{9,3,\chi_{-3};1}(n) q^n ~= ~\eta^3(z) \eta^{15}(3z), \\
\Delta_{9,3,\chi_{-3};2}(z) := \sum_{n\ge 1} \tau_{9,3,\chi_{-3};2}(n) q^n ~= ~\eta^{15}(z) \eta^{3}(3z).
\end{split}
\end{equation*} 
As in the previous case, by comparing the first few Fourier coefficients, it follows that 
$$
{\mathcal F}_9(z) = \frac{2187}{809} E_{9,{\bf 1},\chi_{-3}}(z) + \frac{27}{809} E_{9, \chi_{-3}, {\bf 1}}(z) + \frac{1119744}{809} \Delta_{9,3,\chi_{-3};1}(z) + \frac{41472}{809} \Delta_{9,3,\chi_{-3};2}(z)
$$
from which we get the required formula for $s_{18}(n)$.\\

\noindent Case (iii): $k=11$.  
In this case the dimension of the corresponding modular forms space is $4$. A basis for the $4$-dimensional vector space $M_{11}(\Gamma_0(3), \chi_{-3})$ is given by $E_{11, {\bf 1},\chi_{-3}}(z)$, $E_{11, \chi_{-3}, {\bf 1}}(z)$, $\Delta_{11,3,\chi_{-3};j}(z)$, $j=1,2$, where 
\begin{equation*}
\begin{split}
\Delta_{11,3,\chi_{-3};1}(z) := \sum_{n\ge 1} \tau_{11,3,\chi_{-3};1}(n) q^n ~= ~E_4(z)  \Delta_{7,3,\chi_{-3}}(z), \\
\Delta_{11,3,\chi_{-3};2}(z) := \sum_{n\ge 1} \tau_{11,3,\chi_{-3};2}(n) q^n ~= ~E_4(3z)  \Delta_{7,3,\chi_{-3}}(z). \\
\end{split}
\end{equation*} 
As before, by comparing the first few Fourier coefficients, it follows that 
$$
{\mathcal F}_{11}(z) = \frac{729}{1847} E_{11,{\bf 1},\chi_{-3}}(z) - \frac{3}{1847} E_{11, \chi_{-3}, {\bf 1}}(z) + \frac{60588}{9235} \Delta_{11,3,\chi_{-3};1}(z) + \frac{545292}{9235} \Delta_{11,3,\chi_{-3};2}(z)
$$
from which we get the required formula for $s_{22}(n)$. 
This completes the proof. 

\subsection{Proof of \thmref{s28}}
We consider the two cases $k=12$ and $k=14$ separately. 

\noindent {\bf The case $k=12$}. 
In this case, we have 
\begin{equation*}
{\mathcal F}_{12}(z)  = \sum_{n\ge 1} s_{24}(n) q^n,
\end{equation*}
and the function ${\mathcal F}_{12}(z)$ belongs to $M_{12}(3)$, which has dimension $5$. 
The following modular forms constitute a basis of $M_{12}(3)$:
\begin{equation}
\left\{E_{12}(z), E_{12}(3z), \Delta(z), E_4(z) \Delta_{8,3}(z), E_6(z) \Delta_{6,3}(z)\right\},
\end{equation}
where 
\begin{eqnarray}
\Delta_{8,3}(z) & =& \sum_{n\ge 1} \tau_{8,3}(n) q^n \nonumber \\
&=& \eta^{12}(z)\eta^4(3z) + 81\eta^6(z)\eta^4(3z)\eta^6(9z)+
18\eta^9(z)\eta^4(3z)\eta^3(9z),\\
\Delta_{6,3}(z) &=& \sum_{n\ge 1} \tau_{6,3}(n) q^n ~=~ \eta^6(z) \eta^6(3z).
\end{eqnarray}
Using the above basis, we have 
\begin{equation}
\begin{split}
{\mathcal F}_{12}(z) & = \frac{1}{730} E_{12}(z) + \frac{729}{730} E_{12}(3z) + \frac{29824}{691} \Delta(z) \\
& \qquad +\frac{1186848}{50443} E_4(z) \Delta_{8,3}(z) + \frac{261344}{50443} E_6(z) \Delta_{6,3}(z).\\
\end{split}
\end{equation}
Comparing the $n$-th Fourier coefficients both the sides, and simplifying, we get the required formula \eqref{24}.

\noindent {\bf The case $k=14$}. 
To get the required expression for $s_{28}(n)$, we need to consider the function 
${\mathcal F}_{14}(z)$ which is a modular form of weight  $14$ for the group 
$\Gamma_0(3)$. The vector space $M_{14}(3)$ is $5$-dimensional and we consider the following basis.
$$
\left\{E_{14}(z), E_{14}(3z), E_8(z) \Delta_{6,3}(z), 
E_6(z) \Delta_{8,3}(z),  \frac{1}{2}\big(3 E_2(3z) - E_2(z)\big) \Delta(z)\right\},
$$ 
where the forms $\Delta_{8,3}(z)$ and $\Delta_{6,3}(z)$ are as defined in the previous case. 
Now, expressing the modular form ${\mathcal F}_{14}(z)$ in terms of the basis elements, we get 
\begin{equation*}
\begin{split}
{\mathcal F}_{14}(z) & = -\frac{1}{2186} E_{14}(z) + \frac{2187}{2186} E_{14}(3z) - \frac{3016}{1093} E_8(z) \Delta_{6,3}(z)\\
& \qquad - \frac{12448}{1093} E_6(z) \Delta_{8,3}(z) + \frac{53632}{1093} \big(3 E_2(3z) - E_2(z)\big) \Delta(z). 
\end{split}
\end{equation*}
The $n$-th Fourier coefficient of the LHS is $s_{28}(n)$ and therefore, by comparing the 
$n$-th Fourier coefficients both the sides, we get the required formula \eqref{28}.

\section*{Acknowledgements}
We have used the open-source mathematics software SAGE (www.sagemath.org) to do our calculations. 
The second author is partially funded by SERB grant SR/FTP/MS-053/2012. He would like to thank HRI, 
Allahabad for the warm hospitality where this work has been carried out.

\end{document}